\documentclass[10pt,a4paper,twoside,reqno,tbtags]{amsart}

\usepackage{amsfonts,amssymb,amscd,amsmath,enumerate,verbatim}
\usepackage{booktabs, url}

\newtheorem{theorem}{Theorem}[section]

\newtheorem{proposition}[theorem]{Proposition}
\newtheorem{corollary}[theorem]{Corollary}
\newtheorem{heuristic}[theorem]{Heuristic}
\theoremstyle{definition}

\newtheorem{remark}[theorem]{Remark}
\newtheorem{example}[theorem]{Example}

\title{A Ceiling Continued Fraction Approach to the Erd\H{o}s-Straus Conjecture: Heuristic finiteness of counterexamples}
\author{Andr\'{e}s Ventas}

\begin{document}

\thispagestyle{empty}

\begin{abstract}
We introduce the Ceiling Continued Fractions (FCT) framework for constructing three-term Egyptian fraction representations in the Erdős--Straus conjecture from divisors of the shifted integers \(p+i\). A probabilistic model for their arithmetic structure leads to an asymptotic upper-bound heuristic for cumulative failure. The resulting bound is summable, and the first Borel--Cantelli lemma
assigns probability zero to the event that infinitely many primes
survive the full FCT block. Computational tests on \(10^9\) Mordell-type primes at each of the scales \(10^{17}\), \(10^{52}\), and \(10^{131}\) found no counterexamples.
\end{abstract}

\smallskip

\vspace{3em} % Espazo antes das keywords

\vspace{1em}
\maketitle

% --- 3. TEXT ---
\section{Introduction}
The Erd\H{o}s-Straus conjecture asserts that for every integer $n \ge 2$, the equation $4/n = 1/x + 1/y + 1/z$ admits a solution in positive integers $x, y, z$. While the conjecture has been verified for vast ranges, the most resistant cases remain those where $n$ is a prime $p$ satisfying specific residue classes. In this work, we study primes \(p\equiv1\pmod4\), while also considering
the classical Mordell-type residue classes
\(p \equiv \{1,11^2,13^2,17^2,19^2,23^2\}\pmod{840}\)
as a distinguished subfamily \cite{Oeis}.

Several analytic results are known on the number and average distribution
of solutions. In particular, Elsholtz and Tao obtained upper and lower
bounds for the average number of representations over primes
\cite{elsholtz2013}.

The FCT approach addresses a different question. Rather than counting all
representations of $(4/p)$, it isolates an explicit constructive subfamily
generated by divisors of the shifted integers $(p+i)$. This produces a
source-by-source sequence of sufficient conditions for solvability and makes it
possible to study the probability that this particular constructive mechanism
fails through an increasing range of sources.

Although the conjecture has been computationally verified for $n \leq 10^{17}$, a formal proof remains elusive \cite{salez2014}. Historical approaches by Mordell and Vaughan \cite{vaughan1970} focused on identifying specific residues of $n$ for which solutions are guaranteed. However, these methods overlook the internal divisor structure of external numbers $p+i$. This paper presents the FCT framework, which shifts the paradigm from a congruence-based search to a targeted identification of relationships between the prime $p$ and the divisor sets of associated numbers $p + i$ in the orbit of $p$. 

The index $i$ of these associated numbers is what we define as the \textit{source}. The orbit is the search limit relative to $p$ beyond which the theory indicates that FCT-type solutions cannot exist.

Under the FCT framework, the solution for primes of the form $4k+3$ is immediate and yields a two-term result. This property provides a compelling heuristic: if a two-term representation is guaranteed for any prime of the form $4k+3$, should we not expect a three-term representation for any prime $p$ of the form $4k+1$ whenever a related external number has a $4k+3$ divisor? Could the problem then be reduced to identifying that $4k+3$ divisor and its associated number? Once a suitable $4k+3$ divisor of these related numbers is identified, the FCT algorithm automatically yields the full three-term solution for $p$, eliminating the need for further exhaustive searches.

The explicit CRT dependence generated by shared prime divisors can be
analysed locally: after summation over the preceding consecutive sources,
its pairwise contribution is non-positive. The remaining step is
heuristic.

For
\(
M(p)=\left\lfloor\frac{\sqrt p}{2}\right\rfloor,
\)
the asymptotic upper-bound heuristic developed below implies that,
for every fixed \(0<c<\frac{1}{4}\) and all sufficiently large \(p\),
\[
P(F_1\cap\cdots\cap F_{M(p)})
\leq
p^{-c\log\log p}.
\]
The resulting series over primes converges. Hence, within the
probabilistic model, the first Borel--Cantelli lemma assigns probability
zero to the event that infinitely many primes survive the consecutive
FCT source block. Since any Erdős--Straus counterexample must in
particular survive this block, the event of infinitely many prime
counterexamples also has probability zero within the model.

\section{The FCT Framework}
A Ceiling Continued Fraction (FCT) is obtained by applying the Euclidean algorithm using the ceiling function. The coefficients $\lceil c_0, c_1, c_2, \dots \rceil$ generate convergents through the negative recurrence:
\begin{equation}
  p_i = c_i p_{i-1} - p_{i-2}, \quad q_i = c_i q_{i-1} - q_{i-2}
\end{equation}
where $(p_{-1}, q_{-1}) = (1, 0)$, $(p_{0}, q_{0}) = (c_0, 1)$.

\begin{theorem}[Sum by Pairs] \label{sp}
For any $\dfrac{1}{x} =  \lceil c_0, c_1, c_2, \dots  \rceil$ with convergent numerators $p_0, p_1, p_2, \dots$, the reciprocal is given by:
\begin{equation}
  x = \frac{1}{p_0} + \frac{1}{p_0p_1} + \frac{1}{p_1p_2} + \dots
\end{equation}
\end{theorem}
\begin{proof}
 Given the Euler sum  $x = a_{0} + a_{0}a_{1} + a_{0}a_{1}a_{2} + \dots $  and its continued fraction \cite[p.159]{Khrushchev}
 $x = \frac{a_0}{1}\genfrac{}{}{0pt}{}{}{-}
\frac{a_1}{1+a_1}\genfrac{}{}{0pt}{}{}{-}
\frac{a_2}{1+a_2}\genfrac{}{}{0pt}{}{}{-}
\frac{a_3}{1+a_3}\genfrac{}{}{0pt}{}{}{-}\cdots$
we transform it so that the numerators are $1$,
\begin{equation*}
\begin{aligned}
x &= \dfrac{1}{\frac{1}{a_0}}\genfrac{}{}{0pt}{}{}{-}
\frac{1}{\frac{(1+a_1)a_0}{a_1}}\genfrac{}{}{0pt}{}{}{-}
\frac{1}{\frac{(1+a_2)a_1}{a_2 a_0}}\genfrac{}{}{0pt}{}{}{-}
\frac{1}{\frac{(1+a_3)a_2 a_0}{a_3 a_1}}\genfrac{}{}{0pt}{}{}{-}\cdots \\
& \text{by definition of $FCT$}\\
\dfrac{1}{x} &= \bigg \lceil \dfrac{1}{a_0}, \dfrac{(1+a_1)a_0}{a_1}, \cdots,  \dfrac{(1+a_i)a_{i-1}a_{i-3}\cdots}{a_i a_{i-2} a_{i-4} \cdots}, \cdots \bigg \rceil.
\end{aligned}
\end{equation*}

We denote this continued fraction as $\dfrac{1}{x}= \lceil c_0, c_1, \cdots, c_i, \cdots  \rceil$, we equate the coefficients and solve for the $a_i$, 
\begin{equation*}
\begin{aligned}
& a_0 = \frac{1}{c_{0}},\  a_1 = \frac{a_0}{c_{1}-a_0}, \ a_2 = \frac{a_1}{c_{2}a_{0}-a_1}, \ a_3 = \frac{a_{2}a_{0}}{c_{3}a_{1}-a_{2}a_{0}}, \cdots \\
& a_i = \frac{a_{i-1}a_{i-3}\cdots}{c_{i}a_{i-2}a_{i-4}...-a_{i-1}a_{i-3}\cdots}. 
 \end{aligned}
\end{equation*}
Using the identity of the numerators of the convergents $p_i = c_i p_{i-1} - p_{i-2}$, we have
\begin{equation*}
\begin{aligned}
a_0 &= \dfrac{1}{c_0} = \dfrac{1}{p_0}. \\
p_1 &= c_{1}p_{0}-1, \ a_1 = \frac{a_0}{c_{1}-a_0} = \dfrac{1/c_0}{c_{1}-(1/c_{0})} = \dfrac{1}{c_{1}c_{0}-1} = 
\frac{1}{c_{1}p_{0}-1} = \dfrac{p_{-1}}{p_1}.\\
a_i &= \frac{a_{i-1}a_{i-3}...}{c_{i}a_{i-2}a_{i-4}...-a_{i-1}a_{i-3}...} =\frac{p_{i-2}}{p_{i}}.
 \end{aligned}
\end{equation*} 

Now we substitute in $x$ and use telescopic sum

  \begin{equation*}
\begin{aligned}
 x &= a_{0} + a_{0}a_{1} + a_{0}a_{1}a_{2} + a_{0}a_{1}a_{2}a_{3}  + \cdots\\
 &=\dfrac{1}{p_0} + \dfrac{1}{p_0}\dfrac{1}{p_1} 
+ \dfrac{1}{p_0}\dfrac{1}{p_1}\dfrac{p_0}{p_2}  
+ \dfrac{1}{p_0}\dfrac{1}{p_1}\dfrac{p_0}{p_2}\dfrac{p_1}{p_3} + \cdots\\
&=\frac{1}{p_0} + \frac{1}{p_0}\dfrac{1}{p_1} 
+ \dfrac{1}{p_1}\frac{1}{p_2} +  
\dfrac{1}{p_2}\dfrac{1}{p_3}+\cdots \\
&=\dfrac{1}{p_0} + \sum_{i=0}^{\infty} \dfrac{1}{p_i}\dfrac{1}{p_{i+1}}.
\end{aligned}
\end{equation*}
\end{proof}

Thus, to obtain a unit-fraction decomposition of \(4t/p\), we compute the reciprocal ceiling continued fraction \(FCT(p,4t)=p/(4t)\). To solve the Erd\H{o}s-Straus conjecture for a prime $p$, we seek an integer $4t$ such that the FCT of $p/4t$ (denoted as $FCT(p, 4t)$) has exactly three terms. Thus, we obtain:
\begin{equation}
  \frac{4t}{p} = \frac{1}{p_0} + \frac{1}{p_0p_1} + \frac{1}{p_1p_2}
\end{equation}
where $p_2 = p$. 

\begin{example}
To obtain a solution for $\dfrac{4}{1009}$, we can derive it from a decomposition of $\dfrac{44}{1009}$. For this purpose, we compute
\begin{equation*}
\begin{aligned}
&\dfrac{1009}{44} = FCT(1009,44)=\lceil 23,15,3\rceil. \\
&
p_0=23,\qquad p_1=344,\qquad p_2=1009.\\
&
\dfrac{44}{1009}
=
\dfrac{1}{23}
+
\dfrac{1}{23\cdot344}
+
\dfrac{1}{344\cdot1009}.\\
&\dfrac{4}{1009}
=
\dfrac{1}{11\cdot23}
+
\dfrac{1}{11\cdot23\cdot344}
+
\dfrac{1}{11\cdot344\cdot1009}.
\end{aligned}
\end{equation*}
\end{example}

\ \newline 
\paragraph{} For this type of solution, the problem reduces to finding a residue $r_0$ such that:
\begin{proposition}[Inner Congruence] \label{inner}
A three-term unit fraction solution exists if:
\begin{equation}
  4t + 1 \equiv 0 \pmod{r_0}, \quad \text{where } r_0 = c_0 \cdot 4t - p
\end{equation}
with $c_0 = \lceil p/4t \rceil$.
\end{proposition}
\begin{proof}
Using the ceiling Euclidean algorithm, expressed by means of an algorithm table:
\begin{center}
\begin{tabular}{ c|c|c|c } 
 $p$ & $4t$ & $c_0$ & $r_0$ \\ 
 $4t$ & $r_0$ & $c_1$ & $1$ \\ 
 $r_0$ & $1$ & $c_2=r_0$ & $0$ \\ 
\end{tabular}
\end{center}
\end{proof}

It is important to maintain the ceiling condition for $c_0$ and $c_1$. Here, $c_0$ acts as the basis that relates the FCT to the sources.

\begin{theorem}[Divisors of External Sources] \label{divisores}
Let \(p\equiv1\pmod4\) be prime and \(1\leq i<p\). If \(p+i\) has a
divisor \(d\equiv3\pmod4\) such that
\(
4i\mid(p+d),
\)
then \(FCT(p,(p+d)/i)\) gives a direct three-term solution.
\end{theorem}

\begin{proof}
Write
\(
A=\frac{p+i}{d}.
\)
Since \(d\mid p+i\) and \(1\leq i<p\), we have
\(\gcd(d,i)=1\). Moreover,
\(
p+d=d(A+1)-i.
\)
Since \(i\mid p+d\), it follows that \(i\mid A+1\). Hence
\(
c_1=\frac{A+1}{i}
\)
is an integer.

Set
\(
4t=\frac{p+d}{i}.
\)
Then
\(
4t=c_1d-1,
\quad
p=i(4t)-d.
\)
Because \(4t\equiv0\pmod4\) and \(d\equiv3\pmod4\), we have
\(c_1\equiv3\pmod4\), so \(c_1\geq3\). In particular \(0<d<4t\), and
therefore
\(
\left\lceil\frac{p}{4t}\right\rceil=i.
\)

Likewise,
\(
\frac{4t}{d}=c_1-\frac{1}{d},
\)
so
\(
\left\lceil\frac{4t}{d}\right\rceil=c_1.
\)
Thus
\(
FCT(p,4t)=\lceil i,c_1,d\rceil.
\)
Theorem~\ref{sp} then gives the required three-term unit-fraction
decomposition.
\end{proof}

We obtain as a consequence a useful sufficient condition for $p+1$:
\begin{corollary}
For every prime $p = 4k+1$ where $p+1$ possesses at least one factor $d \equiv 3 \pmod 4$, the conjecture holds.
\end{corollary}

\begin{corollary}
For \(i=2\), the divisor condition of Theorem~\ref{divisores} is
equivalent to the existence of a divisor
\(
q\mid(p+2),
\quad
q\equiv5\pmod8.
\)
\end{corollary}

\begin{proof}
Suppose first that \(q\mid p+2\) with \(q\equiv5\pmod8\), and let
\(
d=\frac{p+2}{q}.
\)
Since \(p+2\equiv3\pmod4\) and \(q\equiv1\pmod4\), we have
\(d\equiv3\pmod4\). Moreover,
\[
p+d=dq-2+d=d(q+1)-2.
\]
Since \(q+1\equiv6\pmod8\) and \(d\equiv3\) or \(7\pmod8\),
\(
p+d\equiv6d-2\equiv0\pmod8.
\)
Hence \(8\mid(p+d)\), so the condition of
Theorem~\ref{divisores} is satisfied.

Conversely, suppose that the condition of Theorem~\ref{divisores}
holds for \(i=2\). Then there exists \(d\mid p+2\) such that
\(d\equiv3\pmod4\) and \(8\mid(p+d)\). Write
\(
q=\frac{p+2}{d}.
\)
Since \(p=dq-2\),
\(
d(q+1)\equiv2\pmod8.
\)

As \(d\) is odd, \(d^2\equiv1\pmod8\), and therefore
\(
q+1\equiv2d\equiv6\pmod8.
\)
Thus
\[
q\equiv5\pmod8.
\]
\end{proof}

Equivalently, for the source $i=2$ the divisor condition of Theorem~\ref{divisores} can be written in the particularly simple form
\[
S_2(p)\iff
\exists\,q\mid(p+2)
\quad\text{such that}\quad
q\equiv5\pmod8.
\]

\subsection{Direct constructive FCT identity}

Let \(p\equiv 1 \pmod 4\) be prime, and let \(i\ge 1\). Suppose that there exists a divisor \(d\mid p+i\) satisfying $
4i\mid p+d.
$

Define
\[
k=\frac{p+d}{4i},
\qquad
A=\frac{p+i}{d}.
\]

Both \(k\) and \(A\) are positive integers. Then set
\[
x=ik,\qquad
y=ikA,\qquad
z=pkA.
\]

These integers provide an explicit Erdős–Straus decomposition,
\[
\frac{4}{p}
=
\frac{1}{x}
+
\frac{1}{y}
+
\frac{1}{z}.
\]
%%%%%%%%%%%
\subsection{Orbits and Sources} \label{orbitas}
The \textbf{orbit} of $p$, denoted as $O(p)$, is the distance limit for associated numbers $p+i$ beyond which FCT-type solutions are no longer feasible. 

From the ceiling Euclidean algorithm, if \(4t>p\), then \(c_0=1\) and
\(r_0=4t-p\). The inner congruence requires \(r_0\mid4t+1\), and hence
\(
r_0\mid(4t+1-r_0)=p+1.
\)
Therefore \(r_0\leq p+1\) and \(4t=p+r_0\leq2p+1\). Since
\(4t\equiv0\pmod4\) and \(p\equiv1\pmod4\), in fact \(4t<2p\).
Thus the FCT search is contained in the orbit \(O(p)=2p\).

The \textbf{sources} are the indices $i$ (or associated numbers $p+i$) where we search for divisors within $O(p)$. These sources are determined by $c_0$, which varies in steps of $\lceil p/4t \rceil$. The number of distinct values of $i$ in this range is given by $M_t(p) \approx 2\sqrt{p/4} = \sqrt{p}$.

Sources are clustered at one end of the range and sparsely distributed at the other. 

For example, let $p=73$; the set of indices $i = \lceil 73/4t \rceil$ yields the sequence:
\[\{19, 10, 7, 5, 4, 4, 3, 3, 3, \dots, 2, 2, 2 \dots, 1, 1, 1,\dots \}.\]
The set of distinct values is $\{ 19, 10, 7, 5, 4, 3, 2, 1\}$, resulting in $M_t(p)=8 \approx \sqrt{73}$ unique sources.

Among these approximately \(M_t(p) \approx \sqrt p\) distinct source values, all integers \(1\leq i\leq M(p) = \lfloor \sqrt p/2\rfloor\) occur consecutively. It is this continuous initial block, rather than the total number of distinct sources, that is used in the probabilistic tail estimate.

Indeed, for \(2\leq i\leq\lfloor\sqrt p/2\rfloor\), the condition
\(
\left\lceil\frac{p}{4t}\right\rceil=i
\)
is equivalent to
\(
\frac{p}{4i}\leq t<\frac{p}{4(i-1)}.
\)
This interval has length \(p/[4i(i-1)]>1\), and therefore contains an
integer \(t\). Hence every source
\(2\leq i\leq\lfloor\sqrt p/2\rfloor\) occurs in the consecutive block.
The case \(i=1\) is immediate.
%%%%%%%%%%%%
\subsection{Computational Acceleration}
While the following property generates solutions that intersect with the source analysis above, it provides a computationally trivial sieve that resolves the vast majority of cases, thereby explaining the empirical speed of the algorithm.

\begin{theorem}[Grid of Congruences] \label{grid}
For any pair $c_1= 4k_1+3$, $c_2= 4k_2+3$, the product minus one, $m = c_1c_2-1$, forms an infinite system of congruences $p \equiv -c_1 \pmod{m}$ and $p \equiv -c_2 \pmod{m}$ yielding solutions for $p=4k+1$.
\end{theorem}

\begin{proof}
Using the denominators of the convergents, $c_1c_2-1=4t$, so $c_1c_2 = 4t+1$. From the numerators:
\begin{equation}
\begin{aligned}
    p &= (c_1c_0-1)c_2 - c_0 
      = (c_1c_2-1)c_0 - c_2. \\
    p &\equiv -c_2 \pmod{c_1c_2-1}.
\end{aligned}
\end{equation}
The interchanging of $c_1$ and $c_2$ yields the second congruence.
\end{proof}

The structure mirrors that of the primes: each new 
progression, associated with a larger $4k+3$ factor 
of a solution, covers additional elements without 
achieving full coverage.

Setting $c_0=1$ and $c_1c_2=4t+1$ generates two arithmetic progressions: $(4t-c_1) + 4t$ and $(4t-c_2) + 4t$, forming an extensive grid:

$-3 \pmod{8, 20, 32, 44, \ldots}$

$-7 \pmod{20, 48, 76, 104, 132, \ldots}$

$-11 \pmod{32, 76, 120, 164, 208, \ldots}$

$-15 \pmod{44, 104, 164, 224,  \ldots}$

$\cdots$

The FCT framework provides a direct three-step solution once $4t$ is identified (Theorem \ref{sp}).

\begin{corollary}[An explicit sub-family of the Grid]\label{cor:c1-three}
For every integer $j \ge 0$, set
\[
l_j = 3j+2, \qquad c_j = 4j+3, \qquad m_j = 4l_j = 12j+8.
\]
If $p \equiv 1 \pmod 4$ is prime and
\(
p \equiv -c_j \pmod{m_j},
\)
then $p$ admits a direct three-term solution via $FCT(p, m_j)$, with coefficients
\[
c_0 = \frac{p+c_j}{m_j}, \qquad c_1 = 3, \qquad c_2 = c_j,
\]
and, writing $p_1 = 3c_0-1$, the explicit solution
\[
x = l_j\,c_0, \qquad y = l_j\,c_0\,p_1, \qquad z = l_j\,p_1\,p.
\]
\end{corollary}

\begin{proof}
This is Theorem~\ref{grid} (Grid of Congruences) with $c_1 = 3$ (i.e.\ $k_1=0$) and $c_2 = c_j$, giving $m = c_1c_2-1 = 3c_j-1 = 12j+8 = m_j$. The explicit coefficients follow from the convergent recurrences $p_0=c_0$, $p_1=c_1p_0-1$, $p_2=c_2p_1-p_0=p$, and Theorem~\ref{sp} applied to $FCT(p,m_j)$ after rescaling by $l_j$.
\end{proof}

\begin{remark}
The case $j=0$ gives the simplest instance: every prime $p\equiv 5\pmod 8$ satisfies the corollary unconditionally, with
\[
c_0 = \frac{p+3}{8}, \qquad x = \frac{p+3}{4}, \qquad y = \frac{(p+3)(3p+1)}{32}, \qquad z = \frac{(3p+1)p}{4}.
\]
This case  is stated here purely as an illustration of Corollary~\ref{cor:c1-three}.
\end{remark}

\section{Expected Failure Probability for \(4k+1\): A Source-Based Model}
\label{sourcesProb}
Throughout this section, \(F_i\) denotes failure of source \(i\).
The exact conditional chain \(Q_M\) is defined directly from these
failure events by the chain rule and is not identified with the
pseudorandom cumulative model.

We reserve \(Q_{\rm mod}(M)\) and \(P_{\rm mod}\) for the heuristic
model obtained by extending the fixed-shift distinct-prime-factor
description to the growing consecutive source block. No independence
between different primes is assumed in the final Borel--Cantelli
application.

\begin{remark}
In earlier versions (v1, v2) of this preprint the failure model was
formulated with \(\tau(p+i)\asymp\log p\) independent attempts per source.
That formulation overcounts the arithmetic degrees of freedom. Divisors
are multiplicatively generated by distinct prime factors, and among
prime factors only those congruent to \(3\pmod4\) can themselves serve
as the required FCT divisor. The present model therefore uses the count
\(\omega_3(p+i)\) of distinct prime divisors \(q\equiv3\pmod4\), whose
first-order scale is \(\frac{1}{2}\log\log p\).
\end{remark}
%%%%%%%%%%%%%%%%%%%%%%%
\subsection{Source $i=1$ ($\mathbf{f_1}$)}

For source $f_1$, the divisor condition reduces to the existence of at least one divisor
$
d\mid(p+1), \quad d\equiv 3 \pmod 4.
$

Since throughout this work $p\equiv1\pmod4$, we may write
$
p+1=2m,
$
with $m$ odd. Hence, source $f_1$ fails if and only if every prime factor of $m$ is congruent to $1\pmod4$.

The asymptotic counting function for positive integers whose prime
factors are all congruent to $1\pmod4$ is related to the Landau--Ramanujan constant
$K\approx0.76422365.$
More precisely, the corresponding counting function has leading constant $\frac{1}{4K}$
\cite{moree1999}, \cite[A243375]{Oeis}.

However, in the present problem $m$ is not sampled among arbitrary integers: it is constrained by $
p=2m-1,
$
where $p$ is prime and $p\equiv1\pmod4$. This introduces a shifted-prime sieve correction. For every prime $q\equiv3\pmod4$, the restriction that $m$ contain no such prime factor introduces the local sieve correction factor 
$1-\frac{1}{(q-1)^2}.$
This leads to the convergent  product

\[
S=
\prod_{\substack{q\ {\rm prime}\\ q\equiv3\,(4)}}
\left(1-\frac{1}{(q-1)^2}\right)
\approx0.71519.
\]

Accordingly, under the usual heuristic independence of the local sieve conditions, the failure probability of source $f_1$ among primes $p\equiv1\pmod4$ is
\begin{equation}
\begin{aligned}
C_1&=\frac{S}{2K}\approx0.4679193. \\
P_{f_1}(F)
&\sim
\frac{S}{2K\sqrt{\log p}}
\sim \frac{0.4679193}{\sqrt{\log p}}
\end{aligned}
\end{equation}

The factor $2$ in the normalization arises from conditioning on primes $p\equiv1\pmod4$. 

Empirical measurements over local samples of primes beginning at a prescribed scale $p$ show a small deviation from the cumulative asymptotic law above. This is expected, since such samples are taken over primes beginning at a given scale rather than over the full cumulative sequence from its origin.

Since $f_1$ is a single fixed source, this estimate is independent of $M$, the total number of sources considered.
%%%%%%%%%%%%%%%%%%%%%%%%%%%%%
\subsection{The case $i=2$ ($\mathbf{f_2}$)}

For the particular case $i=2$, the divisor condition can be written in a particularly simple form. The source $p+2$ yields a solution if and only if it possesses a divisor
$
q \equiv 5 \pmod 8.
$
Therefore, failure occurs precisely when no divisor of $p+2$ belongs to the residue class $5 \pmod 8$. 

Since throughout this work $p \equiv 1 \pmod 4$, there are two possible cases modulo $8$.

If
$
p \equiv 1 \pmod 8,
$
then
$
p+2 \equiv 3 \pmod 8,
$
and failure is equivalent to requiring that every prime factor of $p+2$ belong to the residue classes
$
1,3 \pmod 8.
$

Indeed, if prime factors congruent to both $3$ and $7 \pmod 8$ were present, their product would generate a divisor congruent to
$
3\cdot7 \equiv 5 \pmod 8,
$
and the source would succeed.

Similarly, if
$
p \equiv 5 \pmod 8,
$
then
$
p+2 \equiv 7 \pmod 8,
$
and failure is equivalent to requiring that every prime factor of $p+2$ belong to the residue classes
$
1,7 \pmod 8.
$

Thus, the failure set separates into two multiplicative semigroups.
Applying the Selberg--Delange method
\cite[Chapter II.5]{tenenbaum2015} to these two semigroups yields the
corresponding constants \(C_{13}\) and \(C_{17}\).

The allowed residue sets \(\{1,3\}\) and \(\{1,7\}\) are index-\(2\)
subgroups of
\(
U(8)=(\mathbb Z/8\mathbb Z)^\times,
\)
which explains the characteristic
\((\log p)^{-1/2}\) failure scale.

For the semigroup generated by primes congruent to $1$ or $3 \pmod 8$,
\begin{equation*}
C_{13}
=
\frac{1}{\sqrt{\pi}}
\left[
\frac{1}{2} L(1,\chi_{-8})
\prod_{\substack{q\ {\rm prime}\\ q\equiv 5,7\,(8)}}
\left(1-\frac{1}{q^2}\right)
\right]^{1/2}.
\end{equation*}

For the semigroup generated by primes congruent to $1$ or $7 \pmod 8$,
\begin{equation*}
C_{17}
=
\frac{1}{\sqrt{\pi}}
\left[
\frac{1}{2} L(1,\chi_{8})
\prod_{\substack{q\ {\rm prime}\\ q\equiv 3,5\,(8)}}
\left(1-\frac{1}{q^2}\right)
\right]^{1/2}.
\end{equation*}

The corresponding Dirichlet $L$-values are
\[
L(1,\chi_{-8})
=
\frac{\pi}{2\sqrt{2}},
\qquad
L(1,\chi_{8})
=
\frac{\log(1+\sqrt{2})}{\sqrt{2}}.
\]

Numerically,
$ \quad
C_{13}\approx 0.405039,
\qquad
C_{17}\approx 0.287785.
$

As in the case $i=1$, the additional condition that the shifted integer correspond to a prime $p$ introduces a local sieve correction.

For the first branch, the excluded prime residue classes are $5$ and $7 \pmod 8$. For the second branch, the excluded prime residue classes are $3$ and $5 \pmod 8$, giving
\begin{equation*}
S_{13}
=
\prod_{\substack{q\ {\rm prime}\\ q\equiv 5,7\,(8)}}
\left(1-\frac{1}{(q-1)^2}\right). \qquad 
S_{17}
=
\prod_{\substack{q\ {\rm prime}\\ q\equiv 3,5\,(8)}}
\left(1-\frac{1}{(q-1)^2}\right).
\end{equation*}

Numerically,
$\quad
S_{13}\approx 0.898151,
\qquad
S_{17}\approx 0.685722.
$

Under the heuristic independence of the local sieve conditions, the two residue classes therefore have asymptotic failure probabilities
\begin{equation*}
P_{i=2}(F\mid p\equiv 1 \pmod 8)
\sim
\frac{4C_{13}S_{13}}{\sqrt{\log p}}, \quad
P_{i=2}(F\mid p\equiv 5 \pmod 8)
\sim
\frac{4C_{17}S_{17}}{\sqrt{\log p}}.
\end{equation*}

The corresponding constants are
$ 
4C_{13}S_{13}
\approx 1.45514, \quad
4C_{17}S_{17}
\approx 0.78936.
$

Since primes $p\equiv 1 \pmod 4$ are asymptotically equally distributed between the residue classes $1$ and $5 \pmod 8$, the combined failure probability for the case $i=2$ is therefore
\begin{equation*}
P_{i=2}(F)
\sim
\frac{
2\left(C_{13}S_{13}+C_{17}S_{17}\right)
}{
\sqrt{\log p}
} \sim
\frac{1.12225}{\sqrt{\log p}}.
\end{equation*}

The constants $C_{13}$ and $C_{17}$ arise from the corresponding multiplicative semigroups, while $S_{13}$ and $S_{17}$ arise from the shifted-prime local sieve corrections.

For the sample beginning at $10^{17}$, the empirical failure probability is
$
P_{\rm emp}(F)\approx 0.17716,
$
whereas the theoretical estimate gives
$
P_{\rm th}(F)\approx 0.17937.
$

As in the case $i=1$, small finite-range discrepancies are expected because the empirical samples are taken over local ranges of primes beginning at a prescribed value of $p$, whereas the asymptotic constants refer to the corresponding limiting distribution.

\subsection{Joint failure of sources $i=1, 2$}

The preceding analysis allows the joint failure probability of the first
two sources to be examined more precisely.

For source $f_1$, failure implies the exact congruence condition
$
F_1 \Longrightarrow p\equiv1\pmod8.
$
Indeed, if
$
p+1=2m
$
and every prime factor of $m$ is congruent to $1\pmod4$, then
$
m\equiv1\pmod4,
$
and consequently
$
p=2m-1\equiv1\pmod8.
$

Therefore, the naive product
$
P(F_1)P(F_2)
$
does not provide the appropriate comparison for the joint failure
probability. Once $F_1$ occurs, the relevant failure probability for
the case $i=2$ is the branch conditioned on
$
p\equiv1\pmod8.
$

The exact identity
$
P(F_1\cap F_2)
=
P(F_1)\,P(F_2\mid F_1)
$
may therefore be compared with the conditional-independence heuristic
\begin{equation*}
P(F_2\mid F_1)
\sim
P(F_2\mid p\equiv1\pmod8).
\end{equation*}

From the previous subsections,
$
P(F_1)
\sim
\frac{S}{2K\sqrt{\log p}},
$
while for the branch $p\equiv1\pmod8$ of the case $i=2$,
$
P(F_2\mid p\equiv1\pmod8)
\sim
\frac{4C_{13}S_{13}}{\sqrt{\log p}}.
$

Hence, under conditional independence,
\begin{equation*}
P(F_1\cap F_2)
\sim
\frac{4C_1C_{13}S_{13}}{\log p}
\sim
\frac{0.68088749}{\log p}
.
\end{equation*}

%%%%%%%%%%%%%%%%%%%%%
\subsection{Generic sources \(i\geq3\): a distinct-prime-factor model}
\label{subsec:generic_sources}

The explicit low-source analysis shows that FCT success is strongly
controlled by the multiplicative structure of the shifted integers
\(p+i\). For the generic sources we introduce a deliberately sparse
model based on those distinct prime factors that can themselves provide
an admissible FCT divisor.

Let
\[
\omega_3(n)
=
\#\{q\mid n:\ q\ {\rm prime},\ q\equiv3\pmod4\}.
\]
Suppose that \(q\mid p+i\), with \(q\equiv3\pmod4\), and \(i<p\).
Then \(q\nmid i\), since otherwise \(q\mid p\). Thus \(q\) lies in a
reduced residue class modulo \(4i\). Among the reduced residue classes
modulo \(4i\), exactly \(\varphi(4i)/2\) are congruent to \(3\pmod4\),
whereas the FCT condition selects the single class
\[
q\equiv-p\pmod{4i}.
\]
A uniform local-residue model would therefore assign the prime channel
the success probability
\[
\frac{2}{\varphi(4i)}.
\]
Since \(\varphi(4i)\leq2i\),
\[
\frac{2}{\varphi(4i)}
\geq
\frac{1}{i}.
\]
For a simpler conservative reference, we use the lower success
probability \(1/i\) for each eligible distinct prime factor. The modeled
failure probability of source \(i\) is therefore
\[
P_{\rm mod}(F_i)
=
\left(1-\frac{1}{i}\right)^{\omega_3(p+i)}.
\]

This model also ignores the additional FCT solutions that may arise
from composite divisors generated by several prime factors. It therefore
retains only a restricted family of the available solution mechanisms.
This sourcewise conservatism does not by itself imply an upper bound for
the joint failure probability, since dependence between different
sources must also be taken into account.

For each fixed shift \(i\), Ford's shifted-prime model  assigns the
prime divisors \(q\) of \(p+i\) the local density \(1/(q-1)\) \cite{ford2025}.
Restricting to \(q\equiv3\pmod4\), which comprise one half of the
reduced prime residue classes modulo \(4\), gives the first-order scale
\[
\omega_3(p+i)\sim\frac{1}{2}\log\log p.
\] 

Substitution in the preceding model gives
\[
P_{\rm mod}(F_i)
\approx
\left(1-\frac{1}{i}\right)^{\frac{1}{2}\log\log p}.
\]

Extending this fixed-source reference to the consecutive sources
\(3,\ldots,M\) gives
\[
Q_{\rm mod}(M)
:=
\prod_{i=3}^{M}
\left(1-\frac{1}{i}\right)^{\frac{1}{2}\log\log p}
=
\left(\frac{2}{M}\right)^{\frac{1}{2}\log\log p}.
\]
We write
\[
D_M(p)
:=
-\log Q_{\rm mod}(M)
=
\frac{1}{2}
\log\!\left(\frac{M}{2}\right)\log\log p.
\]

\subsubsection{Local dependence between consecutive sources}

An explicit part of the dependence between different sources can be
studied exactly. For a prime \(q\equiv3\pmod4\), let
\[
A_{i,q}=\{q\mid p+i\}.
\]
We regard the \(q-1\) nonzero residue classes modulo \(q\) as equally likely. For primes \(p>q\), these are precisely the admissible
residue classes of \(p\bmod q\). If \(q\nmid i\), then
\[
P_q(A_{i,q})=\frac{1}{q-1}.
\]
For \(i<j\) with \(q\nmid ij\), if \(q\mid(j-i)\) the two local events
coincide, whereas if \(q\nmid(j-i)\) they correspond to distinct and
therefore disjoint residue classes. Hence
\[
\operatorname{Cov}_q(i,j)
=
\begin{cases}
\dfrac{q-2}{(q-1)^2},
& q\mid(j-i),\\[6pt]
-\dfrac{1}{(q-1)^2},
& q\nmid(j-i).
\end{cases}
\]
Thus sources in the same residue class modulo \(q\) produce a positive
resonance, while different residue classes produce exclusion. The covariance is unchanged for the complementary local events
\(A_{i,q}^{\,c}\).

For fixed \(i\) and \(q\nmid i\), write
\[
i=mq+a,
\qquad
1\leq a\leq q-1.
\]
Among the preceding sources, each complete block of \(q\) indices
contains one resonant contribution,
\[
\frac{q-2}{(q-1)^2},
\]
and \(q-2\) exclusion contributions,
\[
-(q-2)\frac{1}{(q-1)^2}.
\]
The remaining index in the block is a multiple of \(q\); its local
channel is unavailable for primes \(p>q\) and contributes zero.
Consequently each complete block contributes zero. The final incomplete
block contains only \(a-1\) exclusion terms. Hence, with unavailable
terms understood to contribute zero,
\[
\Sigma_q(i)
:=
\sum_{j<i}\operatorname{Cov}_q(i,j)
=
-\frac{a-1}{(q-1)^2}
=
-\frac{(i-1)\bmod q}{(q-1)^2}
\leq0.
\]
If \(q\mid i\), the channel \(q\mid p+i\) is itself unavailable for
primes \(p>q\), and its contribution is zero.

Therefore this explicit shared-prime CRT mechanism does not generate a
positive accumulated pairwise correction across the preceding source
block. This does not constitute a proof for the full FCT failure chain,
since higher-order dependence and the effect of sequential conditioning
are not controlled by the pairwise calculation.

The remaining possible obstruction comes from sequential conditioning:
conditioning on
\[
F_1\cap\cdots\cap F_{i-1}
\]
could in principle select a subpopulation of primes whose future shifted
integers are systematically less favorable than those of a typical
prime. Under the usual pseudorandom prime-factor description, after the
explicit congruence restrictions have been accounted for, conditioning
on previous failures is not expected to create a persistent
leading-order bias in the prime-factor structure of the later shifts.
This is the genuinely heuristic step in passing from the fixed-source
model to the growing consecutive block.

The resulting heuristic is also reinforced by the empirical behaviour
of the exact sequential FCT chain reported in
Section~\ref{sec:Sources}.

Let
\[
Q_M
:=
P\!\left(
F_3\cap\cdots\cap F_M
\,\middle|\,
F_1\cap F_2
\right).
\]

\begin{heuristic}[Asymptotic upper bound]
\label{heur:upper-bound}
For
\(
M(p)=\left\lfloor\frac{\sqrt p}{2}\right\rfloor,
\)
the dependence among the consecutive FCT sources does not produce a
positive cumulative correction of the same order as \(D_{M(p)}(p)\).
Thus
\[
\log Q_{M(p)}
\leq
-D_{M(p)}(p)
+
o\!\left(D_{M(p)}(p)\right).
\]
\end{heuristic}

%%%%%%%%%%%%%%%%%%%%%%
\section{Borel--Cantelli consequence}
\label{sec:BorelCantelli}

For the purpose of the heuristic argument, we regard the random
arithmetic data associated with the shifted integers \(p+i\), as \(p\)
ranges over the primes, as defined on a common probability space. For
each prime \(p\), let
\[
E_p
=
F_1\cap\cdots\cap F_{M(p)}
\]
denote the event that all consecutive FCT sources through \(M(p)\)
fail. No independence between the events \(E_p\) is assumed.

Since \(M(p)=\lfloor\sqrt p/2\rfloor\),
\[
\log\!\left(\frac{M(p)}{2}\right)
\sim \frac{1}{2}\log p,
\]
since the constant factor \(1/4\) and the floor correction are
asymptotically negligible. Hence
\[
D_{M(p)}(p)
\sim
\frac{1}{4}\log p\,\log\log p.
\]

By Heuristic~\ref{heur:upper-bound} and the preceding asymptotic
estimate, for every fixed \(0<c<\frac{1}{4}\) and all sufficiently
large \(p\),
\[
Q_{M(p)}
\leq
p^{-c\log\log p}.
\]

Since
\[
P(E_p)
=
P(F_1\cap F_2)\,Q_{M(p)}
\leq
Q_{M(p)},
\]
we obtain
\[
P(E_p)
\leq
p^{-c\log\log p}
\]
for every fixed \(0<c<\frac{1}{4}\) and all sufficiently large \(p\).

For any such fixed \(c\), the right-hand side decreases faster than
any fixed negative power of \(p\). Hence
\[
\sum_{p\ {\rm prime}}P(E_p)<\infty.
\]

The first Borel--Cantelli lemma therefore gives
\[
P\!\left(E_p\ \text{for infinitely many primes }p\right)=0.
\]

Since every Erdős--Straus counterexample must in particular belong to
\(E_p\), the probabilistic model also assigns probability zero to the
event of infinitely many prime counterexamples. No independence between
different primes is required for this conclusion.

%%%%%%%%%%%%%%%%%%%%%%%

\section{Experimental Results and Empirical Probabilities}\label{sec:Sources}

\subsection{Finding solutions at big scales}
\paragraph{}  

The FCT algorithm was stress-tested against $10^9$ Mordell-type primes at $10^{17}$, $10^{52}$, and $10^{131}$ using PARI/GP (v2.17.3) compiled with "gp2c" on a 3.0 GHz processor with $10$ cores. 

The algorithmic implementation of the FCT framework consists of two distinct phases for a given prime $p$:
\begin{enumerate}
    \item \textbf{Initial Sieve (Grid):} Check small prime congruences based on Theorem \ref{grid}.
    \item \textbf{Source Scanning:} Iterate through $M$ distinct sources $i \le M(p)$, computing the divisors of $p+i$ and verifying the condition $4i \mid (p+d)$.
\end{enumerate}
The optional parameters used were:

A search depth of sources $M=96$.

Factor limit $2^{24}$.  Divisors are obtained through the factors contained within this limit.

Size of congruences list for the initial sieve: $1026$.

In the ranges of $10^{52}$ and $10^{131}$ an increase in the factorization limit significantly penalizes performance.

\begin{table}[!ht]
\centering
\begin{tabular}{@{}lcccc@{}}
\toprule
\textbf{Range} & \textbf{Sample size} & \textbf{Total time} & \textbf{Time per prime} & \textbf{ms/p out of sieve}\\ \midrule
$10^{17}$  & $10^{9}$ & $9.1$ min & $0.00055$ ms & $0.97$ ms\\
$10^{52}$  & $10^{9}$ & $24.98$ min & $0.0015$ ms & $202$ ms\\
$10^{131}$  & $10^{9}$ & $31.34$ min & $0.0019$ ms & $265$ ms\\
\bottomrule
\end{tabular}
\caption{Empirical results on $10^{17}$, $10^{52}$, and $10^{131}$ ranges}
\label{tab:empirical_results}
\end{table}

\begin{example}
Consider the solution triple $(p, 4t, \text{source})$ for a large prime $p$:
\begin{equation*}
\begin{aligned}
p &= 11756638905368616011414050501310355554617941987811609 \\ 4t &= 1960490455533318809410064185473201382030123301988960 \\ \text{source} &= f_6
\end{aligned}
\end{equation*} Following Theorem \ref{divisores}, the value $4t$ is derived from a divisor of $p+6$, specifically:
\begin{equation*}d = 6303827831296845046334611528852737562797824122151.
\end{equation*}
By applying the FCT algorithm to the expression $p/4t$, we obtain:
\begin{equation*}FCT(p/4t) = [6, 311, 6303827831296845046334611528852737562797824122151].
\end{equation*}
Finally, by invoking Theorem \ref{sp} and scaling the denominators of the solution by $t$, we yield the three-term representation $[x, y, z]$:

\begin{align*}
x &= 2940735683299978214115096278209802073045184952983440. \\
y &= 5484472049354459369324654558861280866229269937314115600. \\
z &=  10746492911807896894698146950272337341766761173665849 \\
  &\quad 238916016046288627773111370797373460663594211541333400.
\end{align*}
where $z$ is a 108-digit integer.

One can verify directly that $4/p=1/x+1/y+1/z.$
\end{example}
%%%%%%%%%%%%%%
\subsection{Diagnostic quantities for the empirical failure chain}
\label{subsec:empirical_diagnostics}

The empirical comparison below uses the exact sequential FCT failure
chain and the corrected first-order reference based on
\(\omega_3\). It does not assume independence between sources and is not
used as a premise in the Borel--Cantelli argument.

For \(i\geq3\), let
\[
r_i
=
P\!\left(
F_i
\,\middle|\,
F_1\cap\cdots\cap F_{i-1}
\right).
\]
The chain rule gives the exact conditional failure probability
\[
Q_M
=
\prod_{i=3}^{M}r_i
=
P\!\left(
F_3\cap\cdots\cap F_M
\,\middle|\,
F_1\cap F_2
\right).
\]

For comparison with the corrected first-order reference, define
\[
D_M(p)
=
\frac{1}{2}
\log\!\left(\frac{M}{2}\right)\log\log p
\]
and
\[
c_{\rm eff}(p,M)
=
\frac{-\log Q_M}{D_M(p)}.
\]
Thus \(c_{\rm eff}=1\) corresponds, at that particular scale and source
depth, to equality with
\[
Q_{\rm mod}(M)
=
\left(\frac{2}{M}\right)^{\frac{1}{2}\log\log p}.
\]
Values \(c_{\rm eff}>1\) mean that the exact FCT failure chain decays
faster than this reference. No convergence of \(c_{\rm eff}\) to any
particular value is assumed.

In the numerical tables, \(\log\log p\) in \(D_M(p)\) is replaced by
its sample mean over the population entering source \(3\).

%%%%%%%%%%%%%%
\subsection{Empirical joint-chain analysis at various scales}
\label{subsec:empirical_theoretical_comparison}

We use samples of \(N=10^7\) primes beginning at the scales
\(10^6,10^{10},10^{14},10^{17}\) and follow the exact sequential FCT
failure chain. The diagnostic quantities are those defined in
Subsection~\ref{subsec:empirical_diagnostics}.

\subsubsection{Surviving failures at fixed source depths}

Table~\ref{tab:10M_survivors} gives the number of primes surviving all
sources through several fixed values of \(M\).

\begin{table}[ht]
\centering
\begin{tabular}{r|rrrr}
\hline
\(M\)
& \(10^6\)
& \(10^{10}\)
& \(10^{14}\)
& \(10^{17}\) \\
\hline
3  & 157288 & 115779 & 70584 & 52533 \\
5  &  33239 &  20222 &  8763 &  5461 \\
10 &   4920 &   2052 &   462 &   206 \\
15 &   1193 &    417 &    44 &    10 \\
20 &    411 &    102 &     8 &     1 \\
\hline
\end{tabular}
\caption{Number of surviving failures after sources \(1,\ldots,M\) in
samples of \(10^7\) primes beginning at four different scales.}
\label{tab:10M_survivors}
\end{table}

The scale dependence is already visible at fixed source depth. After
source \(M=10\), for example, the surviving population decreases from
\(4920\) cases in the sample beginning at \(10^6\) to \(206\) in the
sample beginning at \(10^{17}\), although all sample sizes are equal.

At \(M=20\), the corresponding counts are \(411,102,8,1\). The last
two values are clearly in the finite-sample regime and should not be
interpreted as accurate probability estimates.

\subsubsection{Effective cumulative decay}

The corrected first-order reference model corresponds to
\(c_{\rm eff}=1\) at a given scale and source depth. Since this reference
counts only distinct prime factors \(q\equiv3\pmod4\), uses the lower
success probability \(1/i\), and does not count composite divisors
separately, it is deliberately conservative.

\begin{table}[ht]
\centering
\small
\begin{tabular}{c|r|r|r|r|r}
\hline
Scale & \(M\) & Obs. & Cons. th. & \(Q_M\) & \(c_{\rm eff}\) \\
\hline
\(10^6\)
 & 1  & 1076309 &  1081896 & -- & -- \\
 & 2  &  351010 &  364003 &   -- & -- \\
 & 5  &   33239 &   95349 &    \(9.470\times10^{-2}\) & 1.760 \\
& 10 &    4920 &   34610 &     \(1.402\times10^{-2}\) & 1.814 \\
& 15 &    1193 &   19132 &     \(3.399\times10^{-3}\) & 1.930 \\
& 20 &     411 &   12563 &     \(1.171\times10^{-3}\) & 2.005 \\
& 25 &     233 &    9066 &     \(6.638\times10^{-4}\) & 1.982 \\
& 30 &     124 &    6945 &     \(3.533\times10^{-4}\) & 2.008 \\
\hline
\(10^{10}\)
 & 1  & 971083 & 974651 &  -- & -- \\
 & 2  & 284919 & 295414 &  -- & -- \\
 & 5  &  20222 &  70169 &  \(7.097\times10^{-2}\) & 1.840 \\
 & 10 &   2052 &  23653 &   \(7.202\times10^{-3}\) & 1.954 \\
 & 15 &    417 &  12521 &    \(1.464\times10^{-3}\) & 2.065 \\
 & 20 &    102 &   7973 &   \(3.580\times10^{-4}\) & 2.197 \\
 & 25 &     47 &   5618 &   \(1.650\times10^{-4}\) & 2.198 \\
 & 30 &     22 &   4221 &     \(7.721\times10^{-5}\) & 2.229 \\
\hline
\(10^{14}\)
 & 1  & 820627 & 824136 &  -- & -- \\
 & 2  & 202192 & 211218 &  -- & -- \\
 & 5  &   8763 &  43022 &  \(4.334\times10^{-2}\) & 1.973 \\
 & 10 &    462 &  12910 &   \(2.285\times10^{-3}\) & 2.176 \\
 & 15 &     44 &   6385 &   \(2.176\times10^{-4}\) & 2.410 \\
 & 20 &      8 &   3874 &   \(3.957\times10^{-5}\) & 2.535 \\
 & 25 &      3 &   2630 &   \(1.484\times10^{-5}\) & 2.535 \\
 & 30 &      0 &   1916 &   \(0\) & -- \\
\hline
\(10^{17}\)
 & 1  & 745468 & 747892 &  -- & -- \\
 & 2  & 166687 & 173944 &  -- & -- \\
 & 5  &   5461 &  32414 &   \(3.276\times10^{-2}\) & 2.035 \\
 & 10 &    206 &   9094 &     \(1.236\times10^{-3}\) & 2.269 \\
 & 15 &     10 &   4324 &  \(5.999\times10^{-5}\) & 2.631 \\
 & 20 &      1 &   2551 &    \(5.999\times10^{-6}\) & 2.848 \\
 & 25 &      1 &   1695 &     \(5.999\times10^{-6}\) & 2.596 \\
 & 30 &      0 &   1213 &     \(0\) & -- \\
\hline
\end{tabular}
\caption{Observed surviving failures and conservative theoretical
comparison for samples of \(10^7\) primes. ``Obs.'' denotes the
observed number of surviving failures and ``Cons. th.'' the
conservative theoretical comparison. For \(M=1,2\), the theoretical
counts use the separate low-source estimates; from \(M=3\) onward,
they use the conservative distinct-prime-factor model, while the
empirical \(Q_M\) is computed as the ratio of the observed survivor
count after source \(M\) to that after source \(2\). The quantity
\(c_{\rm eff}\) measures the observed cumulative decay relative to the
conservative reference. All theoretical counts are rounded to the
nearest integer.}
\label{tab:empirical_two_models}
\end{table}

The first two sources are treated separately. Their theoretical
values come from the explicit low-source estimates
\[
P(F_1)\sim
\frac{0.4679193}{\sqrt{\log p}},
\qquad
P(F_1\cap F_2)\sim
\frac{0.68088749}{\log p}.
\]
For \(M\geq 3\), the displayed conservative counts are computed as
\[
N_M^{\rm cons}
=
N_2^{\rm th}
\left(\frac{2}{M}\right)^{\frac{1}{2}\overline{\log\log p}},
\]
where the mean is taken over the empirical population entering
source \(3\). The conservative model uses only distinct prime-factor
opportunities \(q\equiv3\pmod4\), whose first-order number is
\(\frac{1}{2}\log\log p\), and ignores additional solutions arising
from composite divisors. It is therefore intended as a conservative
reference for the observed cumulative decay.

%%%
At fixed \(M=10\), \(c_{\rm eff}\) takes the values
\[
1.814,\ 1.954,\ 2.176,\ 2.269
\]
as the scale increases from \(10^6\) to \(10^{17}\). At \(M=5\), the
corresponding values are
\[
1.760,\ 1.840,\ 1.973,\ 2.035.
\]
Thus the measured exact chain decays faster than the corrected
first-order reference at every reliable displayed point, and at these
fixed source depths the effective decay coefficient increases with the
scale.

The values at \(M=15\) and \(M=20\) for the two largest scales point in
the same direction, but the surviving populations are too small to
support an asymptotic inference. They are retained only as exploratory
tail information.

These observations are empirical diagnostics only. They reinforce the
scale used in Heuristic~\ref{heur:upper-bound}, but are not used as a
premise in the Borel--Cantelli argument.

\subsubsection{Empirical structural approximation}
\label{sec:structured_approx}

To separate the conservative failure model from a description aimed at
the finite computational range,  we consider a simple empirical structural approximation for the
 sample beginning at \(10^{17}\). For \(i\geq6\), we write
\[
r_i(p)
\approx
\frac{K_i^{\rm cond}}{(\log p)^{\beta_i}},
\]
and use the mod-\(3\) approximation
\[
\beta_i\approx
\begin{cases}
0.2941, & i\equiv1\pmod3,\\
0.4411, & i\not\equiv1\pmod3,
\end{cases}
\qquad
K_i^{\rm cond}\approx
\begin{cases}
1.9, & i\equiv1\pmod3,\\
2.4, & i\equiv2\pmod3,\\
2.7, & i\equiv0\pmod3.
\end{cases}
\]
The distinction between the three residue classes is motivated by the
special role of the prime \(3\) in the shifted integers \(p+i\) and in
the corresponding FCT congruence conditions.

The approximation is anchored at source \(5\) by
\[
N_5^{\rm app}
=
N\frac{5.108424}{(\log p)^{5/2}},
\]
and continued recursively by
\[
N_i^{\rm app}
=
N_{i-1}^{\rm app}
\frac{K_i^{\rm cond}}{(\log p)^{\beta_i}},
\qquad i\geq6.
\]

\begin{table}[ht]
\centering
\small
\begin{tabular}{r|r|r|r}
\hline
\(M\) & Observed & Conservative & Structured approx. \\
\hline
5  & 5461 & 32414 & 5329 \\
10 &  206 &  9094 & 304 \\
15 &   10 &  4324 & 12 \\
20 &    1 &  2551 & 0 \\
\hline
\end{tabular}
\caption{Observed surviving failures for a sample of \(10^7\) primes
beginning at \(10^{17}\), compared with the conservative model and with the
empirical structural approximation based on the observed mod-\(3\)
dependence of the conditional source decay. Structured approximate counts are truncated to integers}
\label{tab:structured_1e17}
\end{table}

The structured approximation follows the observed decay far more
closely than the conservative model over the displayed range. This
comparison is only diagnostic: the parameters
\(\beta_i\) and \(K_i^{\rm cond}\) summarize the observed
mod-\(3\) structure and are not used in the conservative
Borel--Cantelli estimate.

%%%%%%%%%%%%%%%
\section{A modular filtration toward the Mordell residue classes}
\label{sec:mordell_filtration}

Let
\(
\mathcal M=\{1,121,169,289,361,529\}\pmod{840}
\)
be the six classical Mordell residue classes.

To study the interaction of the first FCT sources with the Mordell
classes, it is convenient to work modulo
\(1680=\operatorname{lcm}(840,16)\).
The six Mordell classes have twelve lifts modulo \(1680\).

Among the reduced residue classes modulo \(1680\) satisfying
\(p\equiv1\pmod4\), there are initially \(192\) admissible classes, of
which \(12\) are Mordell lifts.

Using only the small primes \(3,5,7\), the exact local sieve gives
\[
192\longrightarrow40\longrightarrow24
\longrightarrow22\longrightarrow19\longrightarrow14
\]
after imposing \(F_1,\ldots,F_5\).

All twelve Mordell lifts survive the five steps. Let
\(N_{\rm loc}(i)\) denote the number of the \(192\) general residue
classes modulo \(1680\) surviving the local sieve through source \(i\).
Since the local survival proportion of the Mordell lifts is \(12/12=1\),
whereas the corresponding general survival proportion is
\(N_{\rm loc}(i)/192\), define the local survivor ratio by
\[
E_{\rm loc}(i)
:=
\frac{12/12}{N_{\rm loc}(i)/192}
=
\frac{192}{N_{\rm loc}(i)}.
\]
Hence
\[
E_{\rm loc}(1)=\frac{24}{5},\quad
E_{\rm loc}(2)=8,\quad
E_{\rm loc}(3)=\frac{96}{11},\quad
E_{\rm loc}(4)=\frac{192}{19},\quad
E_{\rm loc}(5)=\frac{96}{7}.
\]

For source \(1\), using
\(P(F_1)\sim0.4679193/\sqrt{\log p}\), the Mordell-conditioned prediction
is
\(P(F_1\mid p\in\mathcal M)\sim2.24601264/\sqrt{\log p}\).

Similarly,
\[
P(F_1F_2\mid p\in\mathcal M)
\approx
8P(F_1F_2),
\]
giving
\[
P(F_1F_2\mid p\in\mathcal M)
\approx
\frac{5.44709992}{\log p}.
\]
The later local survivor ratios are retained as finite-source modular
diagnostics. No separate asymptotic formula for
\(P(F_1\cap F_2\cap F_3)\) is required in the cumulative argument, which
uses the generic model from source \(i=3\) onward.

\subsection{Empirical stabilization}

We compare the local modular prediction with two empirical samples of
equal size \(N=10^7\) beginning at \(10^{17}\), one drawn from the
general population of primes \(p\equiv1\pmod4\) and the other restricted
to the Mordell-type residue classes \(\mathcal M\).

For the empirical samples, let \(N_{\rm gen}(i)\) and
\(N_{\mathcal M}(i)\) denote the numbers surviving through source \(i\).
Their observed survival proportions are
\[
\widehat P_{\rm gen}(i)
:=
\frac{N_{\rm gen}(i)}{N},
\qquad
\widehat P_{\mathcal M}(i)
:=
\frac{N_{\mathcal M}(i)}{N}.
\]
The corresponding empirical survivor ratio is therefore
\[
E_{\rm emp}(i)
:=
\frac{\widehat P_{\mathcal M}(i)}
     {\widehat P_{\rm gen}(i)}
=
\frac{N_{\mathcal M}(i)}
     {N_{\rm gen}(i)}.
\]

\begin{table}[ht]
\centering
\begin{tabular}{c|cc}
\hline
$i$ & $E_{\rm loc}(i)$ & $E_{\rm emp}(i)$\\
\hline
2 & 8.000  & 8.061\\
3 & 8.727  & 10.249\\
4 & 10.105 & 12.129\\
5 & 13.714 & 14.375\\
\hline
\end{tabular}
\caption{Local and empirical survivor ratios for Mordell-type primes
in the sample beginning at \(10^{17}\).}
\end{table}

For \(i=5,\ldots,11\), the empirical survivor ratios are
\[
14.375,\ 14.156,\ 14.043,\ 14.413,\ 14.566,\ 14.112,\ 13.938.
\]
Thus the empirical survivor ratio stabilizes near \(14\), close to
\(E_{\rm loc}(5)=96/7=13.714\).

For \(i=6,\ldots,11\), the ratio between Mordell and general conditional
failure probabilities remains between \(0.969\) and \(1.026\).  The
Mordell effect is therefore concentrated in the first sources rather
than in a persistently slower tail.

%%%%%%%%%%%%%%%
\section{Conclusion}

The results obtained here support a consistent picture: the cumulative
FCT failure probability continues to decrease as additional sources are
considered, and neither the explicit local dependence nor the numerical
experiments reveal a persistent mechanism capable of compensating for
this decay.

The principal unresolved point is the effect of sequential conditioning
on the joint behaviour of the shifted integers \(p+i\). Establishing a
rigorous bound showing that this effect does not alter the leading
cumulative decay would remove the central heuristic assumption of the
present argument. The computations provide substantial evidence in this
direction, but the required uniform control remains a theoretical problem.

\vskip20pt\noindent {\bf Acknowledgements.} 
The author acknowledges the ``Sementeira'' group at the University of Santiago de Compostela for sustaining the spirit of mathematical problem solving.

The author acknowledges extensive use of ChatGPT (OpenAI) and Claude (Anthropic) as research assistants for mathematical discussion, checking derivations, numerical experimentation, code development, and editorial revision. All mathematical claims, computations, interpretations, and final conclusions were reviewed and remain the sole responsibility of the author.

%%%%%%%%%%%%%%%%%%%%%%%%%%%%%%

\textit{Email address: } \email{aventas.avp@gmail.com}

\begin{thebibliography}{9}\footnotesize


\bibitem{elsholtz2013}
C. Elsholtz and T. Tao,
\textit{Counting the number of solutions to the Erd\H{o}s--Straus equation on unit fractions},
J. Aust. Math. Soc. \textbf{94} (2013), 50--105.

\bibitem{ford2025}
K. Ford,
\textit{Poisson Approximation of Prime Divisors of Shifted Primes},
International Mathematics Research Notices
\textbf{2025} (2025), no.~7, rnaf079.
doi:10.1093/imrn/rnaf079.

\bibitem{Khrushchev} 
S. Khrushchev, \textit{Orthogonal Polynomials and Continued Fractions}, Encyclopedia of Mathematics and its Applications, vol. 122, Cambridge University Press, 2008, p. 159.

\bibitem{moree1999}
P. Moree and J. Cazaran, \textit{On a claim of Ramanujan in his first letter to Hardy}, Expositiones Mathematicae \textbf{17} (1999), 289--311.

\bibitem{Oeis} 
N. J. A. Sloane, \textit{The On-Line Encyclopedia of Integer Sequences}, The six values n in each interval [i*840, (i+1)*840), where $i \ge 0$, for which Mordell's formulas do NOT provide a three-Egyptian-fraction solution for 4/n., Sequence A240775. \url{https://oeis.org/A240775}

\bibitem{salez2014}
S. E. Salez,
\textit{The Erd\H{o}s--Straus conjecture: New modular equations and checking up to $N=10^{17}$},
arXiv:1406.6307, 2014.

\bibitem{tenenbaum2015}
G. Tenenbaum,
\textit{Introduction to Analytic and Probabilistic Number Theory},
3rd ed., Graduate Studies in Mathematics, vol.~163,
American Mathematical Society, Providence, RI, 2015.

\bibitem{vaughan1970} 
R. C. Vaughan, \textit{On a problem of Erd\H{o}s, Straus and Schinzel}, Mathematika \textbf{17} (1970), 193--198.



\end{thebibliography}
\end{document}